\documentclass[12pt]{article}
\usepackage[a4paper, total={6in, 8in}]{geometry}
\usepackage{amsmath}
\usepackage{amsfonts}
\usepackage{amsthm}
\usepackage{hyperref}
\usepackage{bbm}
\usepackage{scrextend}
\usepackage{amssymb}
\usepackage{graphicx}
\graphicspath{ {./images/} }
\newcommand*{\p}{\mathbb{P}}
\usepackage{xcolor}
\usepackage{dsfont}
\usepackage{natbib}
\usepackage{graphicx}
\graphicspath{ {./images/} }
\usepackage{float}
\usepackage{multibib}
 \usepackage{algorithm} 
\newcites{SM}{References}

\newcommand{\floor}[1]{\lfloor #1 \rfloor}

\newcommand{\norm}[1]{\left\lVert#1\right\rVert}
\usepackage{setspace}

\newtheorem{Lemma}{Lemma}[section]

\newtheorem{theorem}[Lemma]{Theorem}

\newtheorem{remark}[Lemma]{Remark}

\usepackage[utf8]{inputenc}

\theoremstyle{definition}

\date{}

\definecolor{darkblue}{rgb}{.1, 0.1,.8}
\definecolor{darkgreen}{rgb}{0,0.8,0.2}
\definecolor{darkred}{rgb}{.8, .1,.1}

\usepackage{mathtools}
\mathtoolsset{showonlyrefs}
\newcommand*{\E}{\mathbb{E}}
\newcommand*{\N}{\mathbb{N}}
\newcommand*{\R}{\mathbb{R}}
\renewcommand{\P }{{\mathbb P}}
\newcommand{\1}{\mathbbm{1}}

\title{Detecting relevant deviations from the white noise assumption for non-stationary time series}

\author{
  {\normalsize Patrick Bastian} \\
{\normalsize  Ruhr-Universit\"at Bochum} \\
{\normalsize  Fakult\"at f\"ur Mathematik} \\
{\normalsize  44780 Bochum, Germany}
}

\begin{document}
\maketitle

\begin{abstract}
    
We consider the problem of detecting  deviations from a white noise assumption in  time series.  Our approach differs from the numerous methods proposed for this purpose with respect to two aspects. First, we allow for non-stationary time series. Second, we address the problem that a white noise test, for example checking the residuals of a model fit,  is  usually not performed because one believes in this hypothesis, but thinks that the white noise hypothesis may be approximately true, because a postulated models describes the unknown relation well.  This reflects a meanwhile classical paradigm of  \cite{Box76} that  ”all models are wrong but some are useful”.   We address this point of view by  investigating if the maximum deviation of the local autocovariance functions from 0 exceeds  a given threshold $\Delta$ that can either be specified by the user or chosen in a data dependent way.

The formulation of the  problem in this form raises several mathematical challenges, which do not appear when one is testing the classical white noise hypothesis. We use high dimensional Gaussian approximations for dependent data to furnish a bootstrap test, prove its validity and showcase its performance
on both synthetic and real data, in particular we inspect log returns of stock prices and show that our approach reflects some observations of \cite{Fama1970}
regarding the efficient market hypothesis.
\end{abstract}


\section{Introduction}
\def\theequation{1.\arabic{equation}}	
  \setcounter{equation}{0}
In the last decades the analysis of time series has received an enormous amount of attention on account of its numerous applications including subjects such as economics, hydrology, climatology, engineering, genomics and linguistics to name but a few. One key characteristic of such data that plays both a huge role in theory as well as in applications is their temporal dependence structure, there are two main approaches to address these dynamics. One in the spectral domain, based on the spectral density and one in the time domain, based on the autocovariance function. Independently of the approach a key problem is the testing for a white noise assumption which is important when exploring the serial dependence structure of a time series and when diagnosing whether or not a time series model is a good fit for the data at hand.  We refer the interested reader to \cite{Li2003}, \cite{Francq2010} and \cite{Kim2023} for surveys on this topic for univariate, vector valued and functional time series, respectively. \\

In this paper we are chiefly concerned with a time domain approach to univariate time series, although the results we present can be extended to multivariate or functional data in a fairly straightforward manner. Our central endeavor is to extract information about the autocovariance function $\gamma:\N\rightarrow \R$ from a given sample $X_1,...,X_n$, in particular one is interested in ascertaining if the sequence $\{\gamma(h)\}_{h \in \N}$ is zero, i.e. if the serial correlation structure is trivial. This is, for instance, related to the estimation of the order of linear processes (see \cite{Han2014} and the references therein) and to the goodness of fit of a time series model (see chapter 8 in \cite{box2015}). There is a large amount of statistical procedures available for this testing problem, such as the Box Pierce portmanteau test (and modifications thereof, see \cite{Box1970} and \cite{Ljung78}) which employs a sum of squares test statistic. Another approach consists of analysing the maximal deviation of the empirical autocovariance from its population version, i.e. 
\begin{align}
\label{s1}
    \sqrt{n}\max_{1 \leq h \leq d_n}|\hat \gamma(h)-\gamma(h)|
\end{align}
for some $d_n$ that may tend to infinity. In the stationary case results in this direction can be found in \cite{Hannan1974}, \cite{Jirak2011}, \cite{Han2014} and \cite{Braumann2021} among others. \\

One common thread that unites all the hitherto cited references is that they assume stationarity of the time series, an assumption that is not always valid in practical applications but that significantly simplifies technical considerations. One way of extending the scope of autocovariance analysis that has gained a lot of traction in the literature is the notion of locally stationary time series. We refer the reader to \cite{Zhou2010} and \cite{Dahlhaus2013} and the references therein for additional literature. In this framework the autocovariance function $\gamma$ depends on two parameters $h \in \N$ and $t \in [0,1]$ and is defined by
\begin{align}
\label{p7}
    \gamma(h,t)=\lim_{k/n \rightarrow t} \E[X_{k+h}X_k] \quad h \in \N, t \in [0,1]
\end{align}
where existence is ensured by the definition of locally stationary time series. As mentioned earlier the principal statistical problem we are concerned with in this paper are white noise tests, i.e. testing the hypotheses
\begin{align}
\label{h1}
    H_0:\sup_{h \in \N}\sup_{t \in [0,1]}|\gamma(h,t)|=0 \quad \text{ vs } \quad H_1:\sup_{h \in \N}\sup_{t \in [0,1]}|\gamma(h,t)|>0
\end{align}
Note that for locally stationary time series one may have, for all $h$, $\gamma(h)=0$ for the classical autocovariance while $H_0$ is violated. To the best of our knowledge research on the local autocovariance function and hypotheses of the form \ref{h1} are, compared to their stationary counterparts, still in their nascent stages. We refer the interested reader to the three works \cite{Killick20}, \cite{Cui2021} and \cite{Ding2024}. The first two works define estimators for $\gamma(h,t)$ which are shown to be consistent under Gaussian and non Gaussian assumptions, respectively. \cite{Ding2024} defines a Box-Pierce type statistic, derives its asymptotic (data dependent) distribution, defines a bootstrap procedure and establishes its validity.  \\

We now come to the first key point of this paper, while the assumption of a white noise process is technically convenient it is rare to be satisfied exactly in applications. For instance, "Since all models are wrong" (\cite{Box76}), it seems overly optimistic to apply hypotheses of the form \eqref{h1} to residuals when diagnosing goodness of fit of a time series model. In many cases it might be more reasonable to expect the autocovariances to be negligibly small but not zero and one might still be interested in working under a white noise assumption (or diagnosing the model fit to be acceptable) as long as the deviation from that assumption is not too large. \\
An important example occurs when considering the efficient market hypothesis which implies independence of log returns in certain scenarios, \cite{Fama1970} notes that "Looking hard, though, one can probably find evidence of statistically "significant" linear dependence in Table 1 [...]. But with samples of the size underlying Table 1 [...] statistically "significant" deviations from zero covariance are not necessarily a basis for rejecting the efficient markets model" (see also \cite{Campbell97} for a discussion about testing the efficient market hypothesis that comes to similar conclusions - "Although many of the techniques covered in these pages are central to the market-efficiency debate [...] we feel that they can be more profitably applied to measuring efficiency rather than to testing it."). \\
In these cases a consistent test for \eqref{h1} will, with sufficiently high sample size, detect arbitrary small deviations from the null - rejecting them even when the deviation might be negligible for the application at hand. Motivated by this we propose to instead consider relevant hypotheses, i.e. hypotheses of the form
\begin{align}
\label{h2}
    H_0(\Delta):\max_{h \in \N, t}|\gamma(h,t)|\leq \Delta \quad \text{ vs } \quad H_1(\Delta):\max_{h \in \N, t}|\gamma(h,t)|>\Delta
\end{align}
where $\Delta>0$ that signifies an acceptable amount of deviation from the white noise assumption that can be pre-specified by the user. Of course one may also choose $\Delta$ in an automatic fashion by the sequential rejection principle (see remark \ref{pr3} for a data adaptive procedure that chooses $\Delta$), suggesting the use of $\Delta$ as a measure of evidence against a white noise assumption. Similar perspectives have been taken in \cite{Bücher2023} and \cite{Delft2024} and we refer to the discussions therein for additional details. In section \ref{s3} we showcase the utility of this approach by inspecting the daily log returns of the S\&P500, to be precise we show that $H_0(\Delta)$ is rejected only for very small $\Delta$ indicating that the efficient market hypothesis might be reasonable at least for this particular data set. This stands in contrast to the results of the usual methods like the Box-Pierce test which reject $H_0(0)$ with very small p values. \\

Forgoing the hypotheses \eqref{h1} in favor of testing \eqref{h2} brings with it substantial technical complications. Deriving the asymptotic distribution of an appropriate test statistic under \eqref{h2} is more complicated than under \eqref{h1} as the data is neither stationary nor uncorrelated. In fact it is not clear if a limit distribution exists at all. \\

\textbf{Our Contributions:} 
We extend the relevant hypothesis approach for white noise testing proposed in \cite{Bücher2023} and \cite{Delft2024} to the supremum distance, i.e. we propose a supremum type statistic for hypotheses of the form \eqref{h2}. To be precise we establish a testing procedure that additionally allows for non-stationarity under the null and a diverging number of lags. To the best of our knowledge there are as of yet no results on supremum type statistics in the non-stationary setting and no results on supremum type statistics in a relevant hypothesis framework even for stationary data. This work therefore can be understood as complementary to \cite{Ding2024} in two senses:
\begin{enumerate}
    \item We propose and analyse a supremum type statistic instead of a sum of squares type statistic for the autocovariance of locally stationary data.
    \item We investigate relevant instead of "classical" hypotheses.
\end{enumerate}

From a technical perspective we establish the validity of a block bootstrap procedure for high dimensional $\beta$-mixing time series that is similar to the one provided in \cite{Chernozhukov2018} but only requires the choice of one instead of two block size parameters. Additionally our conditions on the coordinate-wise variances of the data that is to be bootstrapped are substantially less strict.

\textbf{Additional Notations}
For two real valued sequences $(a_n)_{n \in \N}, (b_n)_{n \in \N}$ we use the notation  $ a_n \lesssim b_n$ if $a_n \leq Cb_n$ for some constant $C$ that does not depend on $n$ and that may change from line to line. In addition we write $a_n \simeq b_n$ if $a_n \lesssim b_n$ and $b_n \lesssim a_n$. 
We say that an event $\mathcal{A}$ holds with high probability if $\p(\mathcal{A})=1-o(1)$.

\section{Univariate Time Series}
\def\theequation{2.\arabic{equation}}	
  \setcounter{equation}{0}
In this section we consider a real valued mean zero time series $X_1,...,X_n$ and want to test if is close to a white noise sequence, i.e we are interested in detecting relevant deviations of their local autocovariances from zero. To that end we assume that 
\begin{align}
\label{a2}
    \E[X_{j+h}X_j]=\gamma(h,j/n)+O(d_n/n)
\end{align}
for some function $\gamma(h,\cdot)$ and sequence $d_n \rightarrow \infty$. We define
\begin{align}
\label{a3}
    d_\infty=\sup_{h \in \N, t \in [0,1]}|\gamma(h,t)|
\end{align}
as a measure of deviation of our time series from a white noise process.
We thus consider the hypotheses 
\begin{align}
\label{h11}
     H_0(\Delta):d_\infty \leq \Delta \quad \text{ vs } \quad H_1(\Delta):d_\infty>\Delta
\end{align}
where $\Delta>0$ is a given constant.

\begin{remark}
    \begin{enumerate}
        \item [1)] The assumption that  $X_1,...,X_n$ have mean zero is purely for convenience's sake as including estimation of the means does not change the proofs beyond some cumbersome but straightforward calculations. One may use local (linear) regression to obtain an estimate for the possible time-varying mean function and the results of this paper will apply without changes.
        \item [2)] A large class of time series fulfilling these assumptions is given by piecewise locally stationary time series as given in \cite{Zhou2013}.             
        \item [3)] One may instead of course consider the local autocorrelation function
        \begin{align}
                \rho(h,t)=\gamma(h,t)/\gamma(0,t)
        \end{align}
        and the associated quantities analogous to \eqref{a3} and \eqref{h11}. The methods presented in this paper can be extended to this setting by the obvious modifications of the test statistics and the bootstrap procedure we present. 
    \end{enumerate}
\end{remark}

To construct a test statistic we will require an estimator for $\gamma(h,t)$. For this purpose  let $K$ denote a kernel function and define $K_h(\cdot) =K(\cdot / h)$ for $h>0$. For fixed $t \in [0,1]$ and $h \in \N$ the local linear estimator $\hat \gamma_{h}^{h_n}$ of the function $\gamma_h$ at  the point $t$ with positive bandwidth $h_n=o(1)$ as $n \rightarrow \infty$ is defined by the first coordinate of the
minimizer
\begin{align}
    \label{p401}
    (\tilde \gamma_h^{h_n}(t), (\tilde \gamma_h^{h_n})^\prime(t))=\underset{c_0,c_1 \in \R}{\arg \min}\sum_{j=1}^{n-h}( Y_{jh}-c_0-c_1(j/n-t))^2K_{h_n}(j/n-t).
\end{align}
In the following we will use a bias corrected version (see \cite{Schucany1977}) of the local linear estimator given by \begin{align}
    \label{p17}
    \hat \gamma_h^{h_n}=2\tilde \gamma_h^{h_n/\sqrt{2}}-\tilde \gamma_h^{h_n}~.
\end{align}
When the choice of $h_n$ is clear from the context we will suppress the dependence of $\hat \gamma_h=\hat \gamma_h^{h_n}$ on $h_n$. Theorem  \ref{t3} in the appendix describes the theoretical properties of the resulting estimate.\\

With these estimators at our disposal we can now define, for some sequence $d_n \rightarrow \infty$ the test statistic
\begin{align}
\label{p3b}
    \hat T_{n,\Delta}=\sqrt{nh_n}\max_{\substack{1 \leq h \leq d_n \\ t \in I_{n,h}}}(|\hat \gamma(h,t)|-\Delta):=\sqrt{nh_n}(\hat d_{\infty,n}-\Delta)
\end{align}
where 
\begin{align}
    I_{n,h}=[h_n,1-h_n] \cap [0,1-h/n] ~.
\end{align}
The restriction to the sets $I_{n,h}$ serves two purposes - avoidance of boundary effects in local linear regression and ensuring well definedness of the random variables $Y_j(h)$. \\

Clearly we want to reject $H_0(\Delta)$ for large values of $\hat T_{n,\Delta}$. Obtaining critical values for this statistic is not a trivial matter even in the stationary case. To wit neither the existence of a limit distribution is evident nor, in the event of its existence, is it clear what the appropriate scaling is. We also observe that the asymptotic distribution of statistics of the form \eqref{p3} usually depends on certain extremal sets which are data dependent, further complicating the issue.\\ 

We will propose a solution to these issues by means of a bootstrap procedure which we define in section \ref{s2.4} and whose assumptions are stated in the next Section.  \\

\subsection{Assumptions}
\label{assumptions}
We now proceed with the necessary definitions and assumptions that facilitate the theoretical analysis that will follow.

For modeling the dependence structure of the time series $X_1,...,X_n$ we will use the concept of $\beta$-mixing as introduced in \cite{Bradley2005}. To be precise, assume that all random variables are defined on a probability space $(\Omega, \mathcal{A}, \mathbb{P})$, consider two sigma fields $\mathcal{F}$, $\mathcal{G} \subset \mathcal{A}$ and define
\begin{align}   
    \beta(\mathcal{F},\mathcal{G})=\sup\frac{1}{2}\sum_{i=1}^I\sum_{j=1}^J|\p(F_i\cap G_j)-\p(F_i)\p(G_j)|
\end{align}
where the supremum is taken over all partitions $\{F_1,...,F_I\}$ and $\{G_1,...,G_J\}$ of $\Omega$ such that $F_i \in \mathcal{F}$ and $G_i \in \mathcal{G}$. Next we denote by   $\mathcal{F}_k^{k^\prime}$  the sigma field generated by the process $\{X_j\}_{k \leq j \leq k^\prime}$ and define the $\beta$ mixing coefficients of the sequence $\{X_j\}_{j \in \N}$ by
\begin{align}   
    \beta(k)=\sup_{l \in \N} \beta(\mathcal{F}_1^l,\mathcal{F}_{l+k}^\infty)
\end{align}
The sequence $\{X_j\}_{j \in \N}$ is called $\beta$-mixing if  $\beta(k)\rightarrow 0$ as $k \rightarrow \infty$.

\begin{enumerate}
    \item  [{\bf(M)}]  The time series $\{X_j\}_{j \in \N}$ has bounded moments of order $2J+\delta$ for some $J\geq 8$ and $\delta>0$, i.e.
    $$
    \sup_{j \in \N}\E[X_j^{2J+\delta}]\leq K<\infty.
    $$    
    \item  [{\bf(D)}]
          $\{X_j\}_{j \in \N}$  is $\beta$-mixing, with mixing coefficients $\beta(j)$ satisfying
        \begin{align*}
            \sum_{j=1}^\infty (j+1)^{J-1}\beta(j)^{\frac{\delta}{2J+\delta}}<\infty~, 
        \end{align*}
        for $\delta, J$ from (A1).    
    
    \item [{\bf(K1)}] We have a kernel $K$ that is a symmetric and twice differentiable function, supported on the interval $[-1,1]$, that satisfies $\int_{-1}^{1} K(x)dx=1$ as well as $K(0)>0$. Further assume that the bandwidth $h_n$ satisfies $h_n\sim n^{-\gamma_1}$ where $1/7<\gamma_1 \leq 1/5$
    
    \item [{\bf(K2)}] The partial derivatives $\frac{\partial^2}{\partial t^2}\gamma_h(t)$ of the functions $\gamma_h:[0,1]\rightarrow \R$ exist on $[0,1]$ and are uniformly bounded and Lipschitz continuous, i.e.
    \begin{align}
        \Big |\frac{\partial^2}{\partial t^2}\gamma_h(t)-\frac{\partial^2}{\partial t^2}\gamma_h(t')\Big | &\leq C_1|t-t'|~\\
        |\gamma_h(t)|&\leq C_2
    \end{align}    
    where the constants  $C_1,C_2$ do not depend on $h$. In particular we have that the $\gamma_h$ are Lipschitz with Lipschitz constant $H$ that does not depend on $h$ or $n$.  
   
\end{enumerate}

Assumptions $(M)$ and $(D)$ are standard moment and mixing assumptions, $J\geq 8$ is mostly required to make statements easier to read, it can be relaxed at the cost constraints on other parameters which are difficult to state in an organized manner due to mutual influences on their constraints. Assumptions $(K1)$ and $(K2)$ are standard assumptions for local linear regression (compare to \cite{buecher21} for instance).

\subsection{A bootstrap based on a stochastic expansion for $ \hat T_{n,\Delta}$}
\label{s2.4}
To obtain a bootstrap procedure that produces asymptotically valid quantiles for the statistic $\hat T_{n,\Delta}$ defined in \eqref{p3b} we require a suitable stochastic expansion of $\hat d_{\infty,n}$. We will need a number of additional definitions, beginning with a population version of $\hat d_{\infty,n}$ given by
\begin{align}
    d_{\infty,n}=\max_{\substack{1 \leq h \leq d_n \\ t \in I_{n,h}}}|\gamma(h,t)|~,
\end{align}
as well as the (discretized) extremal sets associated to $d_{\infty,n}$ and $\hat d_{\infty,n}$ which are defined as
\begin{equation}
\label{p201}
\begin{split}
    \hat{\mathcal{E}}^\pm_{\rho}&=\big\{(h,t)\big|  \pm \hat \gamma(h,t) \pm \rho \geq \hat d_{\infty,n}, 1 \leq h \leq d_n, t \in \{1/n,...,1\}\cap I_{n,h}\big\}   \\
\mathcal{E}^\pm_{\rho}&=\big\{(h,t) \big|  \pm  \gamma(h,t) \pm \rho \geq  d_{\infty,n},  1 \leq h \leq d_n, t \in \{1/n,...,1\}\cap I_{n,h}\big\}   ~,.
\end{split}
\end{equation}
Repeated use also makes the following abbreviations useful
 \begin{align}
    &\mathcal{E}:=\mathcal{E}^+_0\cup \mathcal{E}^-_0 \\
    &\hat{\mathcal{E}}:= \hat{\mathcal{E}}^+_0\cup \hat{\mathcal{E}}^-_0\\
     &s(h,t)=\text{sign}(\gamma(h,t))\\
     &\hat s_n(h,t)=\text{sign}(\hat \gamma(h,t))~.
\end{align}

The desired stochastic expansion of $\hat T_{n,\Delta}$ is then given as follows.

\begin{theorem}
    \label{t4}
    Assume that (M), (D), (K1), (K2) are satisfied. Denote by $\rho_n$  any sequence such that $\rho_n^{-1}=o\big ((nh_n)^{1/2}d_n^{-1/J}h_n^{\frac{J+1}{J^2}}\big )$. Further assume that $d_\infty>0$ and that $d_n \lesssim n^{-\gamma}\sqrt{nh_n^{-1}}$. We then have with probability converging to $1$ that
   \begin{align}
   \label{p33}
       \hat d_{\infty,n}-d_{\infty,n} &\leq  \sup_{(h,t) \in \mathcal{E}_{\rho_n}}\Big(\frac{s(h,t)}{nh_n}\sum_{j=1}^{n-h}\epsilon_{jh}K^*_{h_n}(j/n-t)\Big)+ o_\p(n^{-\gamma}(nh_n)^{-1/2}) \\
       & \leq \hat d_{\infty,n}-d_{\infty,n}+O(\rho_n)~, \nonumber 
   \end{align}   
   where
   \begin{align}
       \epsilon_{jh}=X_{j+h}X_j- \gamma_h(j/n)
   \end{align}
\end{theorem}

We now define a block multiplier bootstrap; let $\nu_j$ be a sequence of iid standard normal variables and define for some $l_n \rightarrow \infty$ the multiplier blocks
\begin{align}
    Y_{jh}(t)=l_n^{-1/2}\nu_j\sum_{k=0}^{l_n-1}\hat \epsilon_{(j+k)h}K^*_{h_n}(j/n-t)~.
\end{align}
where $\hat \epsilon_{jh}=X_{j+h}X_j-\hat \gamma_h(j/n)$\\
Now the bootstrap version of the partial sum \eqref{p33} is given as
\begin{align}
\label{bootstat}
    T_n^*:=\sup_{(h,t) \in \hat{\mathcal{E}}_{\rho_n}}(nh_n)^{-1}\hat s_n(h,t)\sum_{j=1}^{n-(l_n \lor h)+1}Y_{jh}(t)~.
\end{align}
for some sequence $\rho_n$ that satisfies $\rho_n^{-1}=o\big ((nh_n)^{1/2}d_n^{-1/J}h_n^{\frac{J+1}{J^2}}\big )$.

We also require one more additional assumption for the bootstrap to work. 
\begin{itemize}
    \item [{\bf(V)}] Let $q=q_n$ be any sequence such that $q_n\rightarrow \infty, q=o(\sqrt{n})$. We assume that for all such sequences we have
    \begin{align}
    \label{pa1}
        \min_{I}\max_{(h,t) \in \mathcal{E}_{\rho_n}}\frac{1}{mh_nq}\sum_{J \in \{I_1,...,I_m\}}\text{Var}\Big(\sum_{ i \in J}X_iX_{i+h}K^*_{h_n}(i/n-t)\Big) \gtrsim 1
    \end{align}
        where $\max_{I}$ is taken over all collections consisting of $m=\lfloor n/q \rfloor$ disjoint subsets of $\{1,...,n\}$ of the form $\{k+1,...,k+q\}$ and $K^*(x)=2\sqrt{2}K(\sqrt{2}x)-K(x)$.
\end{itemize}
\begin{remark}
\begin{enumerate}
    \item [(1)]

    To gain some insight into the nature of this assumption let us assume that $K^*$ is a constant function supported on $[-1,1]$. Then, letting 
    \begin{align}
        L =  \big  \{ \lfloor nt_n  \rfloor - \lfloor \tfrac{ nh_n }{2}  \rfloor  +1, 
         \lfloor nt_n  \rfloor - \lfloor \tfrac{ nh_n}{2}  \rfloor +2, \ldots  , \lfloor nt_n  \rfloor +  \lfloor \tfrac{  nh_n}{2}  \rfloor +1
           \big\} 
    \end{align}
    denote the set of consecutive indices ``centered at $nt_n$  and of size $nh_n$'' we then obtain for any $l \in \{1,...,m\}$ 
    \begin{align}
        \text{Var}\Big(\sum_{ i \in I_l}X_iX_{i+h}K^*_{h_n}(i/n-t)\Big)&\simeq  \text{Var}\Big(\sum_{i \in I_l\cap L}X_iX_{i+h}\Big)
    \end{align}
    For all but two indices $l \in \{1,...,m\}$ we have that that $I_l \cap J$ is either empty or equal to $J_l$. In particular we have that the number of such sets that is nonempty is proportional to $mh_n$. Consequently we obtain that
    \begin{align}
        \frac{1}{mh_nq}\sum_{J \in \{I_1,...,I_m\}}\text{Var}\Big(\sum_{ i \in J}X_iX_{i+h}K^*_{h_n}(i/n-t)\Big) \simeq \frac{1}{mh_n}\sum_{l:I_l \cap J \neq \emptyset} \frac{1}{q}\text{Var}\Big(\sum_{i \in I_l\cap L}X_iX_{i+h}\Big)
    \end{align}
    In the case of a second order stationary sequence this reduces to 
    \begin{align}
        \frac{1}{q}\text{Var}\Big(\sum_{i=1}^qX_iX_{i+h}\Big)
    \end{align}
    which, in the case of strict stationarity, is asymptotically bounded from below whenever
    \begin{align}
        \text{Var}\Big(\sum_{i=1}^{n-h}X_iX_{i+h}\Big)\rightarrow \infty
    \end{align}
    by Lemma 1 from \cite{Bradley1997}.
\item [(2)] One can further weaken this assumption to only concern two particular sequences and two particular  collections of subsets - for details regarding this see the statement of Theorem \ref{pt2} (the two sequences are given by $q$ and $r$ therein, the associated collections of subsets are defined at the beginning of the associated Section).

    \end{enumerate}
\end{remark}

In the following we denote by $\hat q^*_{1-\alpha}$ the $(1-\alpha)$-quantile of the statistic \eqref{bootstat}. We again recall our test statistic
\begin{align}
\label{p101}
    \hat T_{n,\Delta}=\sqrt{nh_n}(\hat d_{\infty,n}-\Delta),
\end{align}
and propose to reject the null hypothesis in \eqref{h11} whenever 
\begin{align}
\label{p92}
    \hat T_{n,\Delta}\geq \hat q^*_{1-\alpha}~.
\end{align}
The main result of this section shows that this decision rule defines a consistent and asymptotic level $\alpha$ test for the hypotheses in \eqref{h11}.

\begin{theorem}
    \label{t5}
      Assume that (M), (D), (K1), (K2), (V) are satisfied. Denote by $\rho_n$  any sequence such that $\rho_n^{-1}=o\big ((nh_n)^{1/2}d_n^{-1/J}h_n^{\frac{J+1}{J^2}}\big )$.  Additionally,  assume that there exists a constant $\delta>0$ such that
    \begin{align}            
        n^{0.325} \lesssim l_n\lesssim n^{1/3} \\        
        d_n \lesssim n^{-\gamma}\sqrt{nh_n^{-1}}
    \end{align} Then the following statements are true.
    \begin{itemize}
    \item [(i)] Under $H_0(\Delta)$ we have  
    \begin{align}
    \label{p44}
         \limsup_{n \to \infty}\p(\hat T_{n,\Delta}>q^*_{1-\alpha})\leq \alpha~.
    \end{align}    
    \item [(ii)] Under $H_1(\Delta)$ we have
    \begin{align}
    \label{p45}
        \lim_{n \to \infty}\p(\hat T_{n,\Delta}>q_{1-\alpha}^{*})=1.
    \end{align}
\end{itemize} 
\end{theorem}

\begin{remark} ~~
\label{pr3}
    \rm 
    \begin{itemize}
        \item[1)] An important question from a practical point of view is the choice of the threshold $\Delta > 0 $, which depends on the specific application under consideration and in some cases might be difficult to determine. In those cases one can also determine a threshold from the data which can serve as measure of evidence against a white noise assumption.
   
    To be precise, note that the hypotheses $H_0(\Delta)$ in \eqref{h11} are nested, that the test statistic \eqref{p3b} is monotone in $\Delta$ and that the quantile  $\hat q^*_{1-\alpha}$ does not depend on $\Delta$. Consequently,  rejecting $H_0(\Delta)$ for $\Delta=\Delta_1$ also implies rejecting $H_0(\Delta)$ for all $\Delta<\Delta_1$. The sequential rejection principle then yields that we may simultaneously test the hypotheses \eqref{p3} for different choices of $\Delta\geq 0$ until we find the minimum value $\hat \Delta_\alpha$  for which $H_0(\Delta_0)$ is not rejected, that is 
    \begin{align}
        \hat \Delta_\alpha:=\min \big \{\Delta \ge 0 \,| \, \hat T_{n,\Delta}\leq \hat q^*_{1-\alpha} \big  \}=\big(\hat d_{\infty,n}-\hat q^*_{1-\alpha}(nh_n)^{-1/2}\big)\lor 0~.
    \end{align}
    Consequently, one may postpone the selection of $\Delta$ until one has seen the data. 
    \item[2)] The restrictions on $l_n$ can be weakened when allowing them to depend on the choices of $d_n$ and $h_n$.
    \end{itemize}
    
\end{remark}

\section{Finite Sample Properties}
\label{s3}
  \def\theequation{3.\arabic{equation}}	
  \setcounter{equation}{0}
In this section we will investigate the finite sample properties of the proposed procedure by means of a simulation study. We also illustrate the new methods by an application to a real data set. 

\subsection{Simulation study}
For the simulation study we consider a stationary and a locally stationary AR(1) process. To be precise we will consider time series of the form
\begin{align}
\label{p49}
    X_i=\phi(i/n)X_{i-1}+\epsilon_j \quad i=1,...,n
\end{align}
where $\{\epsilon_j\}$ is an iid sequence of standard normal variables and $\phi:[0,1]\rightarrow [-1,1]$ is a smooth function. For concreteness we consider the choices
\begin{align} 
\label{p50}
    \phi_1(t)&=0.2\sin(2\pi t)\\
    \label{p61}
    \phi_2(t)&=0.2
\end{align}
for the locally stationary and stationary case, respectively.

Regarding the parameter choices (i.e. $l_n,h_n,\rho_n$) we proceed as follows: 
\begin{enumerate}
    \item[\textbf{(1)}] We utilize the b.star function from the "np" R package \cite{Hayfield2008} to choose the block length $l_n$ for each lag and then average over all lags. This method employs the methods from \cite{Patton2009} to determine the optimal block length for each lag. Averaging isn't strictly necessary but we do it to keep the method in line with the notation in the paper which uses only a single common block length for all lags.
    \item[\textbf{(2)}] We utilize the dpill method from the "KernSmooth" R package separately for each lag and then average over all $d_n$ lags. This method employs a plugin method to determine the optimal bandwidth for estimating each lag. Averaging isn't strictly necessary but we do it to keep the method in line with the notation in the paper which uses only a single common bandwidth for all lags. We have also tried using cross validation which yielded similar results but was computationally more expensive.
    \item[\textbf{(3)}] We choose $\rho_n=0.1\hat \sigma \sqrt{\frac{\log(nh)}{nh}}$ 
    where 
    \begin{align}
        \hat \sigma^2=\frac{1}{d_n}\sum_{i=1}^{d_n}\frac{1}{n-i}\sum_{j=1}^{n-i}(\hat \gamma(h,j/n)-\bar \gamma(h))^2 
    \end{align}
    where $\bar \gamma(h)$ is the mean of $\hat \gamma(h,j/n), j=1,...,n-h$. This follows similar choices as in \cite{Aue2020} and \cite{buecher21} but also incorporates an additional factor that accounts for data specific variations. One could refine this factor further, possibly adjusting the threshold $\rho_n$ used in the estimation of the extremal sets to vary for different $(t,h)$, but this choice yielded reasonable results in simulations. We leave refinement in this direction as a future avenue of research. 
\end{enumerate}

We display the results of the simulations in Figure \ref{Fig:2} below where we show the
empirical rejection probabilities of the test \eqref{p92} for the hypotheses \eqref{h11}  with nominal level $\alpha=0.1$. The vertical line corresponds to the true value $d_\infty $, which is
\begin{align}
    d_\infty=0.2
\end{align}
both in the stationary and the locally stationary case. We choose $n=600$. To obtain the empirical rejection rates we generated 1000 data sets each. We performed our test with $d_n=3,10,20$ lags for each of the functions \eqref{p50}.

\begin{figure}[H]
    \centering

    \includegraphics[scale= 0.34]{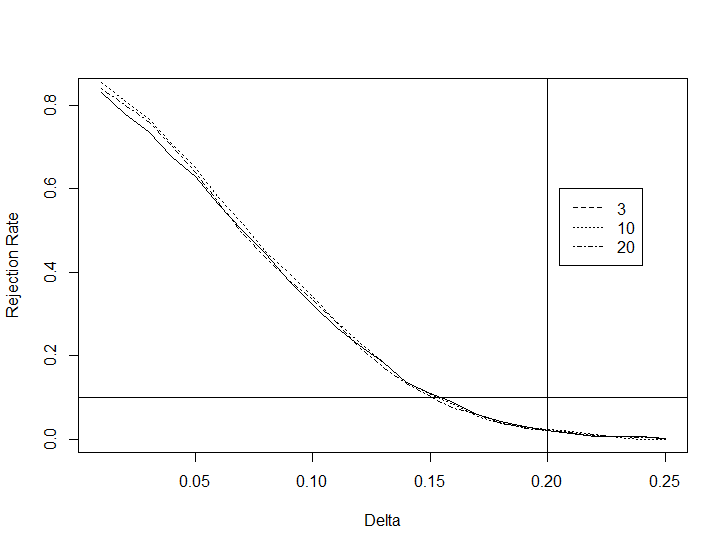}  
    \includegraphics[scale= 0.34]{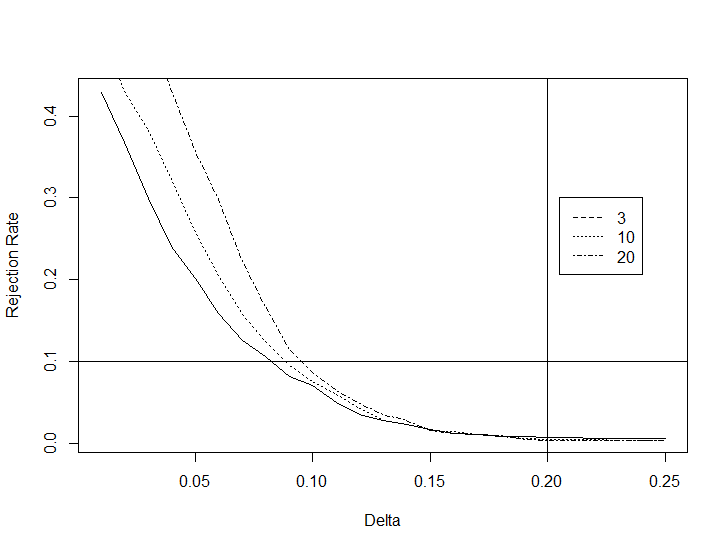}

    \caption{\it    
    Empirical rejection probabilities of the test \eqref{p92} for the hypotheses \eqref{h11}. The sample size is $n=600$, the number of lags is $d_n=3,10,20$ and the maximal correlation is given by $d_\infty=0.2$. The associated time series are given in \eqref{p49} where $\phi$ is given by \eqref{p50} on the left and by \eqref{p61} on the right.
 } 
    \label{Fig:2}
\end{figure}

We observe that the rejection probabilities are decreasing with an increasing threshold $\Delta$, which reflects the fact that it is easier to detect
deviations for smaller thresholds. Note that small values of $\Delta$ correspond to the alternative and that  the power of the test quickly increases as $\Delta$ is decreasing. On the other hand, values of $\Delta \geq d_\infty$ correspond to the null hypothesis and the test keeps its nominal level. These results  confirm the asymptotic  theory in  Theorem \ref{t5}. Comparing the stationary to the locally stationary case we observe that, while the test is conservative in both scenarios, it is much more so in the locally stationary case. This is explained by the fact that estimating the extremal sets accurately is much more difficult in the locally stationary case as $\gamma(h,t)$ is maximized along a sparse set of $t$. The choice of $d_n$ has little impact in the stationary case while larger $d_n$ lead to better detection rates in the locally stationary case.

\subsection{Real data example}
We apply the proposed method to the daily closing prices of the S\&P500 stock index which is available on Yahoo and can be accessed via the quantmod R package \cite{Ryan24}. As is classical in financial data analysis we consider the log returns of the closing prices. We consider the data from 01.01.1980 to 31.12.1999 (i.e. n=5054) whose log returns and (classical) autocorrelation function are displayed in figure \ref{Fig:1}.
\begin{figure}[H]
    \centering

    \includegraphics[scale= 0.34]{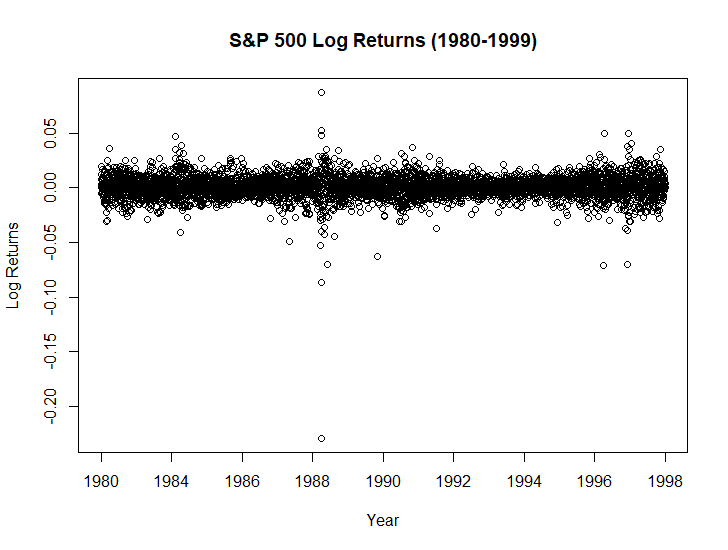}  
    \includegraphics[scale= 0.34]{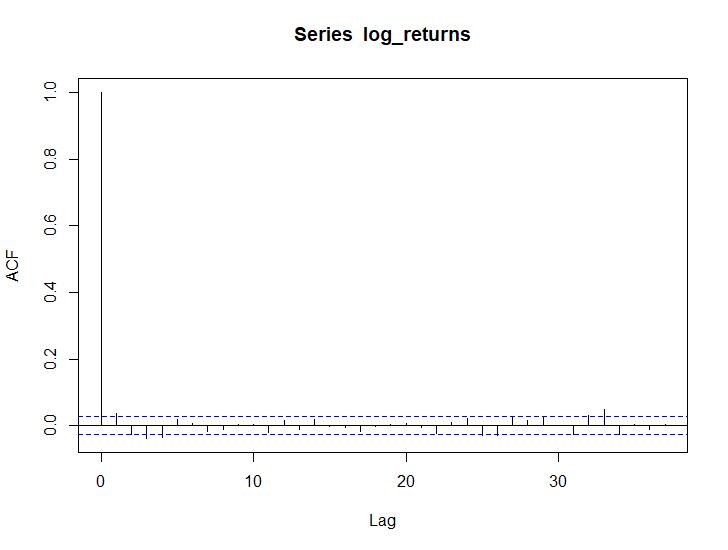}

    \caption{\it    
    Left Side: Log Returns of the S\&P500 between 1980 and 1999. Right Side: Autocorrelation plot of the log returns of the S\&P500 between 1980 and 1999
 } 
    \label{Fig:1}
\end{figure}

Visual inspection immediately yields that the lag correlations are very small, nonetheless standard tests like the Box-Ljung test reject a white noise hypothesis with p values significantly smaller than 0.01 when considering lags up to size 100 and smaller than 0.05 for lags up to size 500. Applying the method proposed in this paper to the local autocorrelation function at level $\alpha=0.05$ yields that $H_0(\Delta)$ is rejected only for $\Delta<0.029$ indicating that the log returns are in fact very close to a (weak sense) white noise sequence. Here we have chosen $l_n, h_n, \rho_n$ in the same way as in the previous section about synthetic data, $d_n$ was set to 3 but similar results are also valid for higher numbers of lags. The extremely small size of $\Delta$ combined with the substantial sample size reflects the view presented in \cite{Fama1970} that "statistically "significant" deviations from zero covariance are not necessarily a basis for rejecting the efficient markets model." Further we also provide a measure of efficiency (the size of $\Delta$) which is much more compatible with "measuring efficiency [of the hypothesis] rather than [...] testing it." (see Section 1.5 in \cite{Campbell97} ) than the usual approach where one investigates hypotheses like \ref{h1}. \\

\textbf{Acknowledgements}
This research was (partially) funded in the course of TRR 391 Spatio-temporal Statistics for the Transition of Energy and Transport (520388526) by the Deutsche Forschungsgemeinschaft (DFG, German Research Foundation).

\newpage

\section{Proofs}
\def\theequation{4.\arabic{equation}}	
  \setcounter{equation}{0}

\subsection{Statement and Proof of  Theorem \ref{t3}}
\begin{theorem}
\label{t3}
Assume that (M), (D), (K1), (K2) (see \ref{assumptions}) hold. We then have that
    \begin{align}        
        &\sup_{ h \leq d_n, t \in I_n}\Big|\tilde \gamma_h(t)-\gamma_h(t)-\frac{1}{nh_n}\sum_{j=1}^{n-h}\epsilon_{jh}K^*_{h_n}(j/n-t)\Big|\\
        = &O(h_n^3+\frac{1}{nh_n}+\frac{d_n}{n}),
   \end{align}
     where $K^*=2\sqrt{2}K(\sqrt{2}x)-K(x)$.   
\end{theorem}
\begin{proof}
Define $\hat \epsilon_{jh}=X_{j+h}X_j-\gamma_h(j/n)$. The desired result follows by the same arguments as in the proof of Lemma C.2 from \cite{Dette2015} combined with \eqref{a2}.
The uniformity in condition 2) of Theorem \ref{t3} serves to ensure that the Taylor approximation errors are uniform in $h$. 
    
\end{proof}

\subsection{Proof of Theorem \ref{t4}}

To make expressions more concise we define the process
\begin{align}
    \mathcal{H}_n(h,t)=\frac{1}{\sqrt{nh_n}}\sum_{j=1}^{n-h}\epsilon_{jh}K^*_{h_n}(j/n-t)~.
\end{align}
We further note that the mixing coefficients $\tilde \beta(k)$ of $\epsilon_{jh}$ satisfy the inequality
\begin{align}
    \tilde \beta(k) \leq \beta(k-h)   
\end{align}
and that similar inequalities also hold for vectors such as $\{(\epsilon_{j1},...,\epsilon_{jd_n})\}_{1 \leq j \leq n-d_n}$. One can extend this to $1 \leq j \leq n$ by appropriately padding the vectors with zeroes.
Further we can always bound the $p$-th moments of $\epsilon_{jh}$ by those of the original time series, i.e.
\begin{align}
    \E[\epsilon_{jh}^p]^{1/p}\lesssim \sup_{k \in \N}\E[X_k^{2p}]^{1/(2p)}
\end{align}
We will make use of of these facts repeatedly without further mention. 

\begin{proof}{ of \textbf{Theorem \ref{t4}}}\\
The Theorem follows immediately by combining Lemma \ref{pl7} with Theorem \eqref{t3}.
\end{proof}

\begin{Lemma}
\label{pl7}
   Let $\rho_n$ be any sequence such that  $\rho_n^{-1}=o\big ((nh_n)^{1/2}d_n^{-1/J}h_n^{\frac{J+1}{J^2}}\big )$. Assume that $d_\infty>0$. We then have that with probability tending to one that 
   \begin{align}
       \hat d_{\infty,n}-d_{\infty,n} &\leq \sup_{(h,t) \in \hat{\mathcal{E}}_{\rho_n}}\hat s_n(h,t)(\hat \gamma(h,t)-\gamma(h,t)))+o_\p(n^{-\gamma}(nh_n)^{-1/2})\\
       & \leq \hat d_{\infty,n}-d_{\infty,n} +O_\p(\rho_n) 
   \end{align}
   The same also holds true when $\hat{\mathcal{E}}_{\rho_n}$ and $\hat s_n(h,t)$ are replaced by $\mathcal{E}_{\rho_n}$ and  $s(h,t)$, respectively.
\end{Lemma}
\begin{proof}
    Using exactly the same arguments  as those in the proof of Lemma 5.2 in \cite{Bastian2024} one obtains the statement without the restriction $t \in \{1/n,...,1\}$, for that Lemma \ref{pl5} and \ref{pl6} are used in place of Lemma 5.4 from the same reference. The desired statement follows by combining Theorem \ref{t3} and equation \eqref{p27} ($\tau=1/(nh_n), \nu=1/n)$  to obtain that the the restriction $t \in \{1/n,...,1\}$ incurs only an error of size $o_\p(n^{-\gamma}(nh_n)^{-1/2})$.
\end{proof}

\begin{Lemma} 
\label{pl4}
    Define
    \begin{align}
        d(t,s):=|t-s|h_n^{-1}~,
    \end{align} then we have for any fixed $1 \leq h \leq d_n$ and any $\tau>0$ that
    \begin{align}
       \sup_{d(t,s)<\tau, t,s \in I_{n,h}}|\mathcal{H}_n(h,t)-\mathcal{H}_n(h,s)| = O_\p(h_n^{-1/J}+\tau h_n^{-2/J})
    \end{align}
    Further we also obtain
    \begin{align}
    \label{p24}
        \E\Big [\sup_{d(t,s)<\tau, t,s \in I_{n,h}}|\mathcal{H}_n(h,t)-\mathcal{H}_n(h,s)|^J\Big ]^{1/J}
        &\lesssim h_n^{-1/J}+\tau h_n^{-2/J}  ~. 
    \end{align}
    where the implied constants do not depend on $\tau$.
\end{Lemma}
\begin{proof}
   Note that
    \begin{align}
        \E[|\mathcal{H}_n(h,t)-\mathcal{H}_n(h,s)|^J]^{1/J}\leq K_1|t-s|h_n^{-1}=K_1d(t,s)~,
    \end{align}    
    which follows by Assumptions (A2) and (A3) and the arguments used for the proof of Theorem 3 from \cite{Yoshi78}. Therefore Theorem 2.2.4 in \cite{Wellner1996} yields that for any $\nu, \tau>0$ and some $K_2>0$ depending only on $J$ and $K_1$ we have 
    \begin{align}   
    \label{p27}
        \E\Big [\sup_{d(t,s)<\tau, t,s \in I_{n,h}}|\mathcal{H}_n(h,t)-\mathcal{H}_n(h,s)|^J\Big ]^{1/J}
        &\lesssim K_2 \big( h_n^{-1/J}\int_0^\nu \epsilon^{-1/J}d\epsilon+\tau h_n^{-2/J}\nu^{-2/J} \big )   \\
        &\lesssim h_n^{-1/J}\nu^{1-1/J}+\tau h_n^{-2/J}   \nu^{-2/J}   
    \end{align}
   Now choose $\nu=1$ and apply Markov to finish the proof.
    
\end{proof}

\begin{Lemma}
\label{pl5}
   We have that 
   \begin{align}
   \label{p26}
        \sup_{1 \leq h \leq d_n, t \in I_{n,h}}|\mathcal{H}_n(h,t)|=O_\p\Big(d_n^{1/J}h_n^{-\frac{J+1}{J^2}}\Big)
   \end{align}
\end{Lemma}
\begin{proof}
   Let $Q_{n,h}$ be partitions of $I_{n,h}$ with width proportional to $\tau$. We then have 
\begin{align*}    
   \sup_{1 \leq h \leq d_n, t \in I_{n,h}}|\mathcal{H}_n(h,t)|
   \lesssim &  \sup_{1 \leq h \leq d_n, t \in Q_{n,h}}\Big|\mathcal{H}_n(h,t)\Big|+\sup_{|t-s|\leq \tau}\Big|\mathcal{H}_n(h,t)-\mathcal{H}_n(h,s)\Big| \\
    \lesssim& \sup_{1 \leq h \leq d_n, t \in Q_{n,h}}\Big|\mathcal{H}_n(h,t)\Big|+d_n^{1/J}O_\p\big (h_n^{-1/J}+\tau h_n^{-1-2/J}\big) 
\end{align*}
where we used Lemma \ref{pl4} for the second inequality in conjunction with Lemma 2.2.2 from \cite{Wellner1996}. Using Theorem 3 in \cite{Yoshi78} combined with Lemma 2.2.2 from \cite{Wellner1996} additionally gives
\begin{align}
    \sup_{1 \leq h \leq d_n, t \in Q_{n,h}}\Big|\mathcal{H}_n(t,s)\Big|= O_\p((d_n|Q_n|)^{1/J})=O_\p(d_n^{1/J}\tau^{-1/J})
\end{align}
which finishes the proof by choosing $\tau=h_n^{1+1/J}$.
\end{proof}

\begin{Lemma}
\label{pl6}
     We have that
    $$\mathcal{E}^+_0 \subset \hat{\mathcal{E}}^+_{\rho_n}~\text{  and
     } ~\hat{\mathcal{E}}^+_0 \subset \mathcal{E}^+_{\rho_n}
     $$ 
     with high probability for any sequence $\rho_n$ such that $\rho_n^{-1}=o\big ((nh_n)^{1/2}d_n^{-1/J}h_n^{\frac{J+1}{J^2}}\big )$. Similar inclusions are valid for the sets $\mathcal{E}^-_0$ and $\hat{\mathcal{E}}^-_0$. \\

     As an immediate consequence we have that $s(h,t)=\hat s_n(h,t)$ with high probability on the sets $\hat{\mathcal{E}}^+_{\rho_n},\hat{\mathcal{E}}^-_{\rho_n}, \mathcal{E}^+_{\rho_n},\mathcal{E}^-_{\rho_n}$ whenever it holds that $d_\infty>0$. 
\end{Lemma}

\begin{proof}
   Follows immediately from Theorems \ref{t3}, Lemma \ref{pl5} and the definition of $\hat \gamma(h,t)$ and $\gamma(h,t)$. 
\end{proof}

\subsection{Proof of Theorem \ref{t5}}
In this section we denote by $\p^*$ the conditional distribution given $X_1,...,X_n$. In many cases auxiliary Theorems and Lemmas will give results involving distributions $\tilde \p^*$ which are given as the conditional distribution given $\{\epsilon_{jh}\}$. In all of these cases it is easy to check that the probabilities in question do not change when swapping from  $\tilde \p^*$ to $\p^*$ and we will use this repeatedly without further notice. We denote by $x_0>0$ some arbitrary but fixed number. 

We define 
\begin{align}
    T_n(\rho_n)=\sup_{(h,t) \in \mathcal{E}_{\rho_n/2}}\Big(\frac{s(h,t)}{\sqrt{nh_n}}\sum_{j=1}^{n-h}\epsilon_{jh}K^*_{h_n}(j/n-t)\Big)
\end{align}
and write $T_n=T_n(\rho_n)$ if the sequence $\rho_n$ is clear from the context.

\begin{proof}

\textbf{Proof of \eqref{p44}}
    We focus on the case $d_\infty>0$, the arguments for $d_\infty=0$ are similar but easier.    

By \ref{t4} we have that
\begin{align}
\label{p40}
    &\sqrt{nh_n}(\hat d_{\infty,n}-d_{\infty,n})\\ \leq &\sup_{(h,t) \in \mathcal{E}_{\rho_n/2}}\Big(\frac{s(h,t)}{\sqrt{nh_n}}\sum_{j=1}^{n-h}\epsilon_{jh}K^*_{h_n}(j/n-t)\Big)+ o_\p(n^{-\gamma})\\
    =& T_n(\rho_n/2) + o_\p(n^{-\gamma}) ~.
\end{align}

By Lemmas \ref{pl8} and \ref{pl9} we have for any $x_0>0$ that 
\begin{align}
\label{p41}
       & \sup_{ y\geq x_0}\Big| \p\Big(T_n(\rho_n/2) \geq y \Big) 
       - \p\Big( \sup_{(h,t) \in \mathcal{E}_{\rho_n/2}}n^{-1/2}\hat s_n(h,t)\sum_{i=1}^{n-(l_n \lor h)+1} Y_{jh}(t)  >y\Big|\mathcal{Y}\Big)\Big|      =o(1)~.
\end{align}

Further we know by the definitions of $\mathcal{E}_{\rho}$ and $\hat{\mathcal{E}}_{\rho}$ and by  Lemma \ref{pl5} that with high probability it holds that $\mathcal{E}_{\rho_n/2}\subset \hat{\mathcal{E}}_{\rho_n}$ which implies that
\begin{align}
\label{p42}
    & \sup_{(h,t) \in \mathcal{E}_{\rho_n/2}}n^{-1/2}\hat s_n(h,t)\sum_{i=1}^{n-(l_n \lor h)+1} Y_{jh}(t)  \\
    \leq& \sup_{(h,t) \in \hat{\mathcal{E}}_{\rho_n}}n^{-1/2}\hat s_n(h,t)\sum_{i=1}^{n-(l_n \lor h)+1} Y_{jh}(t)   
\end{align}
holds with high probability. Denote by $\hat q^*_{1-\alpha}$ and $q^*_{1-\alpha}$ the $(1-\alpha)$-quantiles of 
\begin{align}
     \sup_{(h,t) \in \hat{\mathcal{E}}_{\rho_n}}n^{-1/2}\hat s_n(h,t)\sum_{i=1}^{n-(l_n \lor h)+1} Y_{jh}(t) 
\end{align} and
\begin{align}
   \sup_{(h,t) \in \mathcal{E}_{\rho_n/2}}n^{-1/2}\hat s_n(h,t)\sum_{i=1}^{n-(l_n \lor h)+1} Y_{jh}(t)  ~,
\end{align} 
respectively. By \eqref{p42} we clearly have $q^*_{1-\alpha}\leq \hat q^*_{1-\alpha}$ with high probability.\\

Now combining the above observations we obtain with high probability that 
\begin{align}
    &\p( \hat T_{n,\Delta} > \hat q^*_{1-\alpha})\\
    \leq &\p(  \sqrt{nh_n}(\hat d_{\infty,n}-d_{\infty,n}) > \hat q^*_{1-\alpha}) \\
    \leq& \p(  \sqrt{nh_n}(\hat d_{\infty,n}-d_{\infty,n}) >  q^*_{1-\alpha})+o(1)\\
    \leq& \p\Big(T_n(\rho_n/2)+o_\p(n^{-\gamma})>q^*_{1-\alpha}\Big)+o(1)\\
    =& \p\Big(\sup_{(h,t) \in \mathcal{E}_{\rho_n/2}}n^{-1/2}\hat s_n(h,t)\sum_{i=1}^{n-(l_n \lor h)+1} Y_{jh}(t)>q^*_{1-\alpha}-n^{-\gamma}\Big|\mathcal{Y}\Big)\\
    &~~~~~~~~~+o(1)\\    
    =&\alpha+o(1)~,
\end{align}
where the third inequality follows by \eqref{p40}, the forth follows by \eqref{p41}
and the fifth by Lemma \ref{pl1}.

\textbf{Proof of \eqref{p45}}
Note that
\begin{align}
\label{p48}
    \hat T_{n,\Delta}=\sqrt{nh_n}(\hat d_{\infty,n}-d_{\infty,n})+\sqrt{nh_n}(d_{\infty,n}-\Delta)~.
\end{align}

By Theorem \ref{t4} and Lemma \ref{pl5} we have 
\begin{align}
\label{p46}
    \sqrt{nh_n}(\hat d_{\infty,n}-d_{\infty,n}) \lesssim O_\p\Big(d_n^{1/J}h_n^{-\frac{J+1}{J^2}}\Big)~.
\end{align}
Further we observe by the uniform continuity of $\gamma_h$ that for $n$ sufficiently large we have
\begin{align}
    d_\infty-d_{\infty,n}<(d_\infty-\Delta)/2
\end{align}
so that for $n$ large enough
\begin{align}
\label{p47}
    \sqrt{nh_n}(d_{\infty,n}-\Delta) \geq \sqrt{nh_n}(d_\infty-\Delta)/2~.
\end{align}
Combining \eqref{p48}, \eqref{p46} and \eqref{p47} gives
\begin{align}
     \hat T_{n,\Delta} \geq  O_\p\Big(d_n^{1/J}h_n^{-\frac{J+1}{J^2}}\Big)+ \sqrt{nh_n}(d_\infty-\Delta)/2
\end{align}
which yields the desired result.
\end{proof}

\begin{Lemma}
\label{pl8}
Define 
\begin{align}
    S_n^*=\sup_{(h,t) \in \mathcal{E}_{\rho_n}}n^{-1/2}\hat s_n(h,t)\sum_{i=1}^{n-(l_n \lor h)+1}\tilde Y_{jh}(t)
\end{align}
where $\tilde Y_{jh}(t)=l_n^{-1/2}\nu_i\sum_{k=0}^{l_n-1} \epsilon_{(j+k)h}K^*_{h_n}(j/n-t)$. Then we have that
\begin{align}
    \sup_{t>x_0}\Big| \p(T_n>t)-\p^*(S_n^*>t)\Big|\lesssim o(1)
\end{align}
on a set with probability $1-o(1)$.
\end{Lemma}
\begin{proof}
    We want to apply Theorem \ref{boot}  (for the variables $\{\epsilon_{jh}K^*_{h_n}(j/n-t)h_n^{-1/2}\}$) for which we need to verify its requirements. D1) and D2) as well as $\beta(n)\lesssim n^{-3}$ follow immediately from our assumptions. We are left with verifying $D3)$ for $p \lesssim n^{1.6}$ and with verifying the conditions on $\bar \sigma(q)$ and $\bar \sigma(r)$. Regarding $D3)$ we note that using that $\epsilon_{jh}$ has slightly more than 8 moments we obtain that $e_1$ slightly larger than $1/6$ is a valid choice as $(np)^{1/ 8}\lesssim n^{0.325-\delta}$ for some $\delta>0$. Consequently we can always choose some $e_2>0$ as $\beta(k)\lesssim k^{-8}$ by (D). It follows that any sequence $l_n$ satisfying
    \begin{align}
        n^{0.325}\lesssim l_n \lesssim n^{1/3}
    \end{align}
    is permissible for applying Theorem \ref{boot}.
    The conditions on $\bar \sigma(q)$ and $\bar \sigma(r)$ follow from assumption $(V)$. This yields the desired statement with  
    \begin{align}
        \sup_{(h,t) \in \mathcal{E}_{\rho_n}}n^{-1/2}s(h,t)\sum_{i=1}^{n-(l_n \lor h)+1}\tilde Y_{jh}(t)
    \end{align}
    in place of $S_n^*$.  The desired assertion is then obtained by noting that by Lemma \ref{pl6} we may replace $s(h,t)$ by its estimator without changing the above quantity on a set of high probability.\\

\end{proof}

In the next Lemma we swap out $\epsilon_{jh}$ for their estimators. 
\begin{Lemma}
    \label{pl9} 
Let
\begin{align}
    R_n^*=\sup_{(h,t) \in \mathcal{E}_{\rho_n}}n^{-1/2}\hat s_n(h,t)\sum_{i=1}^{n-(l_n \lor h)+1} Y_{jh}(t)~,
\end{align} 
    then, provided that $lh_n^{-1}\lesssim n^{0.6}$ we have for any sequence $\rho_n$ that 
    \begin{align}
        \sup_{t>x_0}\Big| \p(R^*_n>t)-\p^*(S_n^*>t)\Big|\lesssim o(1)
    \end{align}
    holds on a set of probability $1-o(1)$.
\end{Lemma}
\begin{proof}
   Using Lemma \ref{pl5} we obtain that
    \begin{align}
        \sup_{1 \leq h \leq d_n, j/n \in I_{n,h}}|\epsilon_{jh}-\hat \epsilon_{jh}| \lesssim o_\p(n^{-(3/10+\delta)})
    \end{align}
    holds on a set with probability $1-o(1)$ for some small $\delta>0$. From here on one can proceed as in the proof of Lemma 5.8 in \cite{Bastian2024} to obtain the desired result (the interested reader may use Lemma \ref{pl1} in place of verifying the proposed modification of Theorem 1 from \cite{chernozhukov2017} used in the proof of Lemma 5.8 in the cited reference).    
\end{proof}

\section{Gaussian Approximation and Block-Bootstrap for High Dimensional $\beta$-mixing Data}

\subsection{Bounded Data}

The proofs in this section will rely on the big blocks small blocks technique pioneered by Bernstein. To that end let $q=q_n>r=r_n$ be two positive integers with $2(q+r)<n$ and define
\begin{align}
    &I_1=\{1,...,q\},\\
    &J_1=\{q+1,...,q+r\},\\
    &\quad \quad \quad \vdots \\
    &I_m=\{(m-1)(q+r)+1,...,(m-1)(q+r)+q\},\\
    &J_m=\{(m-1)(q+r)+q+1,...,m(q+r)\},\\
    &J_{m+1} = \{m(q+r),...,n\}
\end{align} where $m=m_n=\floor{n/(q+r)}$. In other words we decompose the sequence $\{1,...,n\}$ into large blocks $I_1,...,I_m$ of length $q$ separated by small blocks $J_1,...,J_{m}$ of length $r$ and a remaining part $J_{m+1}$.

In this section we will consider a mean zero sample $X_1,...,X_n \in \R^p$ that we assume to be $\beta$-mixing. Further let $T_n=\max_{1 \leq j \leq p} n^{-1/2}\sum_{i=1}^nX_{ij}$. Assume that there exists some constant $D_n$ such that 
\begin{align} 
\label{p400}
    |X_{ij}|\leq D_n \quad \text{ a.s. } 1 \leq i \leq n, 1 \leq j \leq p
\end{align}
Further let 
\begin{align}
    B_l=\sum_{ i \in I_l} X_i, \quad S_l= \sum_{ i \in J_l} X_i
\end{align}
and $\{ \tilde B_l , 1 \leq l \leq m\}$, $\{ \tilde S_l, 1\leq l \leq m\}$ be sequences of independent copies of $B_l$ and $S_l$. Additionally let $Y=(Y_1,...,Y_p)$ be a centered normal vector with covariance matrix $\E[YY^T]=(mq)^{-1}\sum_{i=1}^m\E[B_iB_i^T]$. \\

We also define
\begin{align}
    \bar \sigma(q):=\max_{1 \leq j \leq p }\frac{1}{m}\sum_{I \in \{I_1,...,I_m\}}\text{Var}\left(q^{-1/2}\sum_{i \in I}X_{ij}\right) \\
    \bar \sigma(r):=\max_{1 \leq j \leq p }\frac{1}{m}\sum_{J \in \{J_1,...,J_m\}}\text{Var}\left(q^{-1/2}\sum_{i \in J}X_{ij}\right) 
\end{align}
Additionally let $\bar \sigma(0):=1$.\\
Finally we denote for a sigma Algebra $\mathcal{X}$ by $\P^*(\cdot)$ the measure $\P(\cdot| \mathcal{X})$, in the following $\mathcal{X}$ will always be the sigma algebra generated by either  $X_1,...,X_n$ or their truncated versions and the choice will always be clear from the context and hence will not be reflected in the notation.\\
We begin with establishing a Gaussian approximation result for $\beta$-mixing random variables that allows for decaying variances in all but one coordinate.

\begin{theorem}
\label{pt1}
   Suppose that 
   \begin{align}
       \bar \sigma(q) &\gtrsim 1 \\
       \bar \sigma(q) \lor \bar \sigma(r) &\lesssim 1\\
       (q+r\log(pn))\log(pn)^{1/2}D_nn^{-1/2}&\lesssim n^{-c_1}\\
       r/q \log(p)^2 & \lesssim n^{-c_1}\\
       q^{1/2}D_n\log(pn)^5n^{-1/2}&\lesssim n^{-c2}       
   \end{align}
   then
   \begin{align}
     \sup_{t >x_0}\left|\p(T_n\leq t)-\p\Big(\max_{1 \leq j \leq p}Y_j\leq t\Big)\right|\lesssim n^{-c}+(m-1)\beta(r)
\end{align}
for some $c>0$ only depending on $c_1,c_2$. In particular we also have that
\begin{align}
     \sup_{t >x_0}\left|\p(T_n\leq t+\epsilon)-\p(T_n\leq t)\right|\lesssim n^{-c}+(m-1)\beta(r)
\end{align}
\end{theorem}
\begin{proof}
For the remainder of this proof $c$ denotes a generic positive constant that may change from line to line that only depends on $c_1,c_2$ as well as $d_1,d_2$.  First we bound the error incurred by leaving out the small blocks, i.e. we have that
    \begin{align}
        |T_n-\max_{1 \leq j \leq p}n^{-1/2}\sum_{i=1}^mB_{ij}|\leq \max_{1 \leq j \leq p}\Big|n^{-1/2}\sum_{i=1}^mS_{ij}\Big| + \max_{1 \leq j \leq p}\Big|n^{-1/2}S_{(m+1)j}\Big|
    \end{align}
and note that
\begin{align}  
\label{b1}
    \max_{1 \leq j \leq p}\Big|n^{-1/2}S_{(m+1)j}\Big|&\lesssim  qD_n/\sqrt{n}     
\end{align}

By Corollary 2.7 from \cite{Yu1994} we have
\begin{align}
\label{b2}
   \Big| \p\Big(\max_{1 \leq j \leq p}n^{-1/2}\sum_{i=1}^mB_{ij}\leq t\Big)-\p\Big(\max_{1 \leq j \leq p}n^{-1/2}\sum_{i=1}^m\tilde B_{ij}\leq t\Big)\Big|& \leq (m-1)\beta(r)\\
   \Big| \p\Big(\max_{1 \leq j \leq p}n^{-1/2}\Big|\sum_{i=1}^mS_{ij}\Big|\leq t\Big)-\p\Big(\max_{1 \leq j \leq p}n^{-1/2}\Big|\sum_{i=1}^m\tilde S_{ij}\Big|\leq t\Big)\Big| &\leq (m-1)\beta(q)
\end{align}
Using independence and the fact that $\tilde S_{ij}$ is bounded by $rD_n$ we obtain by Lemma D.3 from \cite{Chetverikov2020} that
\begin{align}
    \E\left[\max_{1 \leq j \leq p}\Big|n^{-1/2}\sum_{i=1}^m\tilde S_{ij}\Big|\right]\lesssim \sqrt{r/q \bar \sigma(r)\log(p)}+n^{-1/2}rD_n\log(p) 
\end{align}
so that 
\begin{align}
\label{p3}
     \p\Big(\max_{1 \leq j \leq p}n^{-1/2}\Big|\sum_{i=1}^m\tilde S_{ij}\Big|> t\Big) &\lesssim \left(\sqrt{r/q \bar \sigma(r)\log(p)}+n^{-1/2}rD_n\log(p)\right)/t   
\end{align}
\end{proof}

Let $\delta_1=\sqrt{n^{-c_1}/\log(pn)}$. Combining \eqref{b1}, \eqref{b2}, \eqref{p3} (choose $t=\delta_1$ in \eqref{p3}) we obtain for $t>x_0$ that
\begin{align}
\label{p300}
    \p(T_n\leq t)& \leq  \p\Big(\max_{1 \leq j \leq p}n^{-1/2}\sum_{i=1}^mB_{ij}\leq t+\delta_1+qD_nn^{-1/2}\Big)+n^{-c}+(m-1)\beta(q)\\
    & \leq \p\Big(\max_{1 \leq j \leq p}n^{-1/2}\sum_{i=1}^m\tilde B_{ij}\leq t+\delta_1+qD_nn^{-1/2}\Big)+n^{-c}+(m-1)\beta(q)\\
    & \leq  \p\Big(\max_{1 \leq j \leq p}n^{-1/2}\sum_{i=1}^m\tilde B_{ij}\leq t\Big)+n^{-c}+ \Big(\frac{qD_n^2\log(pn)^{10}}{n}\Big)^{1/6}+(m-1)\beta(r)
\end{align}
where we used Lemma \ref{pl1} for the last line. A similar argument also yields the reverse inequality so that for all $t>x_0$ we have
\begin{align}
\label{p301}
   \sup_{t >x_0}\left|\p(T_n\leq t)-\p\Big(\max_{1 \leq j \leq p}n^{-1/2}\sum_{i=1}^m\tilde B_{ij}\leq t\Big)\right|\lesssim n^{-c}+(m-1)\beta(r)
\end{align}

Now we apply Theorem 4.1 from \cite{Chetverikov2020}  (note that for at least one $j$ we have $\bar \sigma(q) \simeq n^{-1}\sum_{i=1}^m\text{Var}(\tilde B_{ij})$ and that for all $j$ it holds that $\tilde B_{ij}\lesssim qD_n$)  to further obtain that
\begin{align}
     \sup_{t >x_0}\left|\p(T_n\leq t)-\p\Big(\max_{1 \leq j \leq p}\sqrt{mq/n}Y_j\leq t\Big)\right|\lesssim n^{-c}+(m-1)\beta(r)
\end{align}
The result then follows by the same arguments as in step 4 of the proof of Theorem E.1 from \cite{Chernozhukov2018} where we use Lemma C.3 from \cite{Chetverikov2020} instead of Step 3.\\

\begin{Lemma}
\label{pl1}
     Suppose that the conditions of Theorem \ref{pt1} are fulfilled for $q=1, r=0$ and that $X_1,...,X_n$ are independent. Then for for any $x_0>0$ we have that for all $t>x_0$ and $\epsilon>0$ it holds that
    \begin{align}
        \p\Big(\sup_{1 \leq j \leq p}n^{-1/2}\sum_{i=1}^nX_{ij}\leq t+\epsilon\Big)-\p\Big(\sup_{1 \leq j \leq p}n^{-1/2}\sum_{i=1}^nX_{ij}\leq t\Big)\lesssim \Big(\frac{D_n^2\log(pn)^{10}}{n}\Big)^{1/6}+\epsilon\sqrt{\log(pn)}
    \end{align}
\end{Lemma}
\begin{proof}
   The proof is the same as for Lemma 4.4 in \cite{Chernozhukov22} but we use Theorem 4.1, equation (56) and Lemma C.3 from \cite{Chetverikov2020} instead of Lemma 4.3 and J.3.
\end{proof}

We now describe the block multiplier bootstrap which is different from the one defined in \cite{Chetverikov2020}, to be precise we only need to choose one block size instead of two. Let $\nu_j$ be a sequence of iid standard normal variables and define for any $1 \leq l \leq n$ the variables
\begin{align}
    Y_{ij}=l^{-1/2}\nu_i\sum_{k=0}^{l-1}X_{i+k,j}~.
\end{align}
Now the bootstrap version of $T_n$ is given as
\begin{align}
    T_n^*:=\sup_{1 \leq j \leq p}n^{-1/2}\sum_{i=1}^{n-l+1}Y_{ij}~.
\end{align}

We first evaluate how well the conditional covariance matrix of the partial block sums approximates the covariance structure of the original partial sums in the supremum norm, to that end we define $\Sigma$ and $\hat \Sigma$ by
\begin{align}    \Sigma_{kh}&=\E\Big[\Big(n^{-1/2}\sum_{i=1}^nX_{ik}\Big)\Big(n^{-1/2}\sum_{i=1}^nX_{ih}\Big)\Big] \quad 1 \leq k,h \leq p\\
    \hat \Sigma_{kh}&=\E^*\Big[\Big(n^{-1/2}\sum_{i=1}^nY_{ik}\Big)\Big(n^{-1/2}\sum_{i=1}^nY_{ih}\Big)\Big]   \quad 1 \leq k,h \leq p  ~,
\end{align}
respectively.

\begin{theorem}
\label{pt2}
Suppose that $\beta(n)\lesssim n^{-3}$ and that one of the following assumptions holds:
\begin{enumerate} 
    \item[1)] Assume that $\log(pn)\simeq \log(n)$ and that there exists constants $c_1,c_2,c_3>0$ as well as $\delta_3<c_1$ such that
    \begin{align}
        D_n/\sqrt{n}&\lesssim n^{-c_1}  \\
        n\beta(n^{c_1})&\lesssim n^{-c_2} \\
        D_nn^{c_3}&\lesssim l \lesssim n^{2c_1-c_3}    \\
        \bar \sigma(q) &\gtrsim  1 \\
        \bar \sigma(q) \lor \bar \sigma(r) &\lesssim 1
    \end{align}  
where $q=n^{1/2-\delta_3}D_n^{-1}$ and $r=q^{1-\delta_4}$.
         
    \item[2)] Assume that $\log(pn)\simeq n^{\gamma}$ for some $\gamma>0$ and that there exists constants $d_1,d_2,d_3,d_4>0$ and $0<2\delta< d_1 \land (1/2-5\gamma/2)$ such that
    \begin{align}
        n^{\gamma-1}D_n^2&\lesssim n^{-d_1}\\
        n^{19\gamma-1-d_1}D_n^2&\lesssim n^{-d_2}\\
        n^{1-d_1+2\delta}\beta(n^{1/2-5\gamma/2-2\delta}D_n^{-1})&\lesssim n^{-d_3}\\
        D_nn^{3/2\gamma+d_4}&\lesssim l=o(n^{1-7\gamma}D_n^{-2})\\
        n\beta(n^{1-7\gamma}D_n^{-2})&=o(1)\\
        \bar \sigma(q) &\gtrsim  1 \\
        \bar \sigma(q) \lor \bar \sigma(r) &\lesssim 1
    \end{align}
    where $q=n^{1/2-\gamma/2-\delta}D_n^{-1}$ for some small $\delta>0$ and $r=q^2D_nn^{-(1+3\gamma)/2}$.
\end{enumerate}
    We then have for any $x_0>0$ that
    \begin{align}
     \sup_{t \geq x_0}\Big|\p(T_n \leq t)-\p^*(T_n^*\leq t)\Big|\lesssim n^{-c_0}
    \end{align}
    holds with probability $1-o(1)$ for some $c_0>0$ that depends only on $c_1,c_2,c_3$ (or $d_1,d_2,d_3,d_4, \delta)$ and the implied constants in the assumptions.
\end{theorem}
\begin{proof}
We first consider the case $\log(pn)\simeq \log(n)$, the case $\log(pn)\simeq n^{\gamma}$ follows by similar considerations. $c$ denotes a constant  that does not depend on $n$ whose value may change form line to line. \\

Combine Lemma \ref{pl2} and \ref{pl3} (choose $t=\log(pn)^{-3}n^{-\delta_1}$ for some $\delta_1>0$ and $\mu\simeq nD_n^{-2}\log(pn)^{-7}$) to obtain that
     \begin{align}
     \label{p10}
        \sup_{t >x_0}\left|\p(T_n^*\leq t)-\p\Big(\max_{1 \leq j \leq p}Y_j\leq t\Big)\right|\lesssim n^{-c}+D_n^{2/3}(q^{-2}+l^{-2}+n^{-1})^{1/3}\log(pn)
     \end{align}
     holds with probability $1-o(1)-n\beta(n^{2c_1-\delta_2}-l)$ for some $c>0$, any $1 \leq q,l \leq n$ and any $\delta_2>0$. 
     
     We now apply Theorem \ref{pt1} with $q=n^{1/2-\delta_3}D_n^{-1}$ and $r=q^{1-\delta_4}$ where $\delta_4>0$ is arbitrary and $0 < \delta_3 < c_1$. This yields
     \begin{align}
     \label{p11}
      \sup_{t >x_0}\left|\p(T_n\leq t)-\p\Big(\max_{1 \leq j \leq p}Y_j\leq t\Big)\right|&\lesssim n^{-c}+D_nn^{1/2+\delta_3}\beta(n^{c_1-\delta_5})   \\
      & \lesssim n^{-c}+n\beta(n^{c_1-\delta_5}) 
     \end{align}
     for any $\delta_5>\delta_3$. Combining \eqref{p10} and \eqref{p11} yields the desired result by the assumption on $l$.\\

     For the other case we choose $\mu$ and $t$ as above, $q=n^{1/2-\gamma/2-\delta}D_n^{-1}, r=q^2D_nn^{-(1+3\gamma)/2}$.

\end{proof}

\begin{Lemma}  
\label{pl2}
    Suppose that $\beta(n)\lesssim n^{-3}$, then some $C>0$ and all $1 \leq l, \mu, q \leq n$ we have that
    \begin{align}
    \norm{\Sigma-\hat \Sigma}_\infty &\lesssim t+O(D_n^2l^{-2})+O(D_n^2n^{-1})\\
    \norm{\Sigma-\E[YY^T]}_\infty &\lesssim O(D_n^2q^{-2})+O(D_n^2n^{-1})
    \end{align}
    holds with probability at least $1-\mu lp^2\exp(-Ct^2\frac{n}{D_n^2\mu})-n\beta(\mu-l)$. 
\end{Lemma}
\begin{proof}
    Define
    \begin{align}
        \sigma_i^{kh}&=n^{-1}\sum_{j=1}^n\E[X_{jk}X_{(j+i)h}]\\
        \hat \sigma_i^{kh}&=n^{-1}\sum_{j=1}^nX_{jk}X_{(j+i)h}
    \end{align}
    and note that
    \begin{align}
        \hat \Sigma_{kh}&=\E^*\Big[\Big(n^{-1/2}\sum_{i=1}^nY_{ik}\Big)\Big(n^{-1/2}\sum_{i=1}^nY_{ih}\Big)\Big]\\
        &=\hat\sigma^{kh}_0+\sum_{i=1}^l \frac{l-i}{l} (\hat \sigma_i^{kh}+\hat \sigma_i^{hk})-  (nl)^{-1}\sum_{f=1}^{l-1}\sum_{i=1}^{l-f-1}(l-i-f)(X_{ik}X_{(i+f)h}+X_{ih}X_{(i+f)k}) 
        \\ & \quad \quad \quad -(nl)^{-1}\sum_{i=1}^{l-1}(l-i)(X_{ik}X_{ih}) \\
        &=\hat\sigma^{kh}_0+\sum_{i=1}^l \frac{l-i}{l} (\hat \sigma_i^{kh}+\hat \sigma_i^{hk})+O(D_n^2n^{-1})~. 
    \end{align}
    Hence, leveraging a similar calculation for $\Sigma_{kh}$, we obtain
    \begin{align}
       \Big| \Sigma_{kh}-\hat \Sigma_{kh}\Big| &\leq |\hat \sigma_0^{kh}-\sigma_0^{kh}|+\sum_{i=1}^l\frac{l-i}{l}|\hat \sigma_i^{kh}+\hat \sigma_i^{hk}-2\sigma_i^{kh}|+2\sum_{i=1}^l\frac{i}{l}|\sigma_i^{kh}|+2\sum_{i=l+1}^\infty |\sigma_i^{kh}|+O(D_n^2n^{-1})\\
       & \leq |\hat \sigma_0^{kh}-\sigma_0^{kh}|+\sum_{i=1}^l\frac{l-i}{l}|\hat \sigma_i^{kh}+\hat \sigma_i^{hk}-2\sigma_i^{kh}|+O(D_n^2l^{-2})+O(D_n^2n^{-1})~.
    \end{align}
    Now consider the vector whose entries are given by
    \begin{align}
        X_{ik}X_{(i+f)h}-\E[ X_{ik}X_{(i+f)h}] \quad \quad \quad 1 \leq f \leq l, 1 \leq k,h \leq p
    \end{align}
    and apply Theorem 1 from \cite{Yoshihara1978} together with standard results on the concentration of Subgaussian random variables to obtain for some $C>0$ and any $1 \leq \mu \leq n$ that
    \begin{align}
        \sup_{1 \leq k,h \leq p}\Big| \Sigma_{kh}-\hat \Sigma_{kh}\Big| &\lesssim  t+O(D_n^2l^{-2})+O(D_n^2n^{-1})
    \end{align}
    holds with probability  at least $1-lp^2\exp\Big(-Ct^2\frac{n}{D_n^2\mu}\Big)-n\beta(\mu-l)$.

    Similarly we have uniformly in $1 \leq k,h \leq p$ that
    \begin{align}
        \E[(YY^T)]_{kh}=\sigma_0^{kh}+2\sum_{i=1}^q\frac{q-i}{q}\sigma_i^{kh}=\Sigma_{kh}+O(D_n^2q^{-2})+O(D_n^2n^{-1}) ~.
    \end{align}
\end{proof}

\begin{Lemma}
\label{pl3}
    Let $Z_1, Z_2$ be two $p$-dimensional mean zero Gaussian vectors. Then for any $x_0>0$ we have
    \begin{align}
        \sup_{t\geq x_0}\Big|\p\Big(\max_{1 \leq j \leq p}Z_{2j}\leq t\Big)-\p\Big(\max_{1 \leq j \leq p}Z_{1j}\leq t\Big)\Big|\lesssim \Delta^{1/3}\log(np/\Delta)
    \end{align}
\end{Lemma}
\begin{proof}
First use the arguments from the proof of Theorem 1 from \cite{Chetverikov2020} to reduce to the case $\E[Z^2_{1j}] \gtrsim \log(pn)^{-1}$. We now only need to make the implicit dependence of the constants on $\underline \sigma$ in the bound of Theorem 2 from  \cite{Chernuzhokov2013b} explicit to obtain the desired result. The dependence of the constants on $\underline \sigma$ is inherited from the use of Theorem 3 in its proof, we can make them explicit by using the bounds from Theorem 1 in \cite{chernozhukov2017} instead.
\end{proof}

\subsection{Unbounded Data}

We now extend the results of the previous subsection to unbounded data by a truncation argument. We only consider polynomial tail  assumptions (and hence at most polynomial growth of $p$) but our results can also be extended to exponential tails (and hence also exponential growth of $p$) with minor modifications. We adopt all notations from the previous subsection, only dropping the assumption that the random variables $X_{ij}$ fulfill \eqref{p400}.

\begin{theorem}
\label{boot}
     In addition to $\beta(n)\lesssim n^{-3}$ we assume 
     \begin{align}
        \bar \sigma(q) &\gtrsim  1 \\
         \bar \sigma(q) \lor \bar \sigma(r) &\lesssim 1
    \end{align}
     where $q=n^{1/2-\delta}(np)^{-1/J}/\log(n)$ for some $\delta>0$ and $r=q^{1-\tilde \delta}$ for some $\tilde \delta>0$.
   Additionally we require 
     \begin{enumerate}
         \item [D1)]          
            \begin{align}    
                  \sup_{1 \leq j \leq p}\E[|X_{1j}|^{J+\delta}]&\lesssim 1       
            \end{align}      
    for some even $J\geq 2$ and some $\delta>0$.
         \item [D2)] For the same $J,\delta$ as in $D1)$  it holds that
         \begin{align}
              \sum_{k=1}^\infty (k+1)^{J/2-1}\beta(k)^{\frac{\delta}{J+\delta}}&<\infty ~.
         \end{align}
         \item [D3)] There exists constants $e_1,e_2,e_3$ such that
         \begin{align}
             (np)^{1/J}n^{-1/2}\log(n)&\lesssim  n^{-e_1}\\
             n\beta(n^{e_1})&\lesssim n^{-e_2}\\
             (np)^{1/J}\log(n)n^{e_3}\lesssim l &\lesssim n^{2e_1-e_3}
         \end{align}
     \end{enumerate}
      Then for any $x_0>0$ we have that that with high probability
     \begin{align}
         \sup_{t \geq x_0}&\Big|\p\Big( \sup_{1 \leq j \leq p} n^{-1/2}\sum_{i=1}^nX_{ij} \leq t\Big) - \p^*\left( \sup_{1 \leq j \leq p}n^{-1/2}\sum_{i=1}^{n-l+1}Y_{ij} \leq t \right) \Big| \\ \lesssim  & o(1)~.
    \end{align}      
\end{theorem}
\begin{proof}
    We want to apply Theorem \ref{pt1}, to that end we first truncate our data and then recenter it. \\
    
    Define for some $u>1$ to be specified later the quantities
    \begin{align}
        \tilde X_{ij}=X_{ij}\1\{X_{ij}\leq u\}\\
        \check X_{ij}=X_{ij}\1\{X_{ij}>u\}
    \end{align}
    and note that
    \begin{align}
        X_{ij}- \tilde X_{ij}=\check X_{ij}~.
    \end{align}

    We will establish the result in three steps.\\
    \textbf{Step 1: Bootstrap Truncation is negligible}\\
    We want to show that the bootstrap and the truncated bootstrap are asymptotically equivalent, i.e.
    \begin{align}    
       \sup_{t \geq x_0}& \left|\p^*\left( \sup_{1 \leq j \leq p}n^{-1/2}\sum_{i=1}^{n-l+1}Y_{ij} \leq t \right)- \p^*\left( \sup_{1 \leq j \leq p}\frac{1}{\sqrt{nl}}\sum_{i=1}^{n-l+1}\nu_l\sum_{k=0}^{l-1}(\tilde X_{(i+k)j}-\E[\tilde X_{(i+k)j}])\leq t \right)\right|\\ &\lesssim o(1)~.
    \end{align} 
    with high probability.
    To that end we first note that by  the union bound  and Markov's inequality that 
    \begin{align}
        \p\Big(X_{ij}\neq \tilde X_{ij} \text{ for at least one pair $(i,j)$}\Big)\lesssim npu^{-J}
    \end{align}
     Using Lemma \ref{pl3} along with calculations similar to those in the proof of Lemma \ref{pl2} gives that on the set $\{X_{ij}=\tilde X_{ij} \ \  \forall (i,j)\}$ we have
     \begin{align}
         \max_{1 \leq k,h \leq p}\Big|\hat \sigma_{kh}-\tilde \sigma_{kh}\Big|&\lesssim \max_{1 \leq j \leq p} n^{-1}\Big|\sum_{i=1}^n\Big( X_{ij}\E[\tilde X_{ij}]+\E[\tilde X_{ij}]^2\Big)\Big| 
     \end{align}
     Due to $\E[X_{ij}]=0$ we obtain that $|\E[\tilde X_{ij}]|=|\E[\check X_{ij}]|$ so that by Hölder's inequality
     \begin{align}
         |E[\tilde X_{ij}]|\lesssim u^{-J/2}
     \end{align}
    which then yields
    \begin{align}
    \label{p23}
         \max_{1 \leq k,h \leq p}\Big|\hat \sigma_{kh}-\tilde \sigma_{kh}\Big|&\lesssim n^{-1}\max_{1 \leq j \leq p} \Big|\sum_{i=1}^n X_{ij}\E[\tilde X_{ij}]\Big| +u^{-J}     
     \end{align}
    For the first term on the right hand side we note that by arguments similar to those in equation \eqref{p21} that
    \begin{align}
    \label{p22}
        \p\Big(\max_{1 \leq j \leq p}\Big|\sum_{i=1}^n X_{ij}\E[\tilde X_{ij}]\Big|>t\Big)\lesssim pu^{-J^2/2}t^{-J}~.
    \end{align}
    We now choose $u=(np)^{1/J}\log(n)$, letting $t=1$ in \eqref{p22} then yields that on a set of probability $1-o(1)$ we can use \eqref{p23} to obtain
    \begin{align}
         \max_{1 \leq k,h \leq p}\Big|\hat \sigma_{kh}-\tilde \sigma_{kh}\Big| \lesssim n^{-1}~.
    \end{align}
    An application of Lemma \ref{pl3} then finishes step 1.\\

    \textbf{Step 2: The change in $T_n$ when truncating the data is negligible}\\
    Observe that 
    \begin{align}
    \label{p20}
        \p\Big( \sup_{1 \leq j \leq p} n^{-1/2}\sum_{i=1}^nX_{ij} \leq t\Big)\geq \p\Big( \sup_{1 \leq j \leq p}n^{-1/2}\sum_{i=1}^n(\tilde X_{ij}-E[\tilde X_{ij})] \leq t-\sup_{1 \leq j \leq p} n^{-1/2}\sum_{i=1}^n(\check X_{ij}-\E[\check X_{ij}]) \Big)
    \end{align}    
    Consider for an $s$ to be chosen later the inequalities
    \begin{align}
     \label{p21}
         \p\Big( \sup_{1 \leq j \leq p} n^{-1/2}\sum_{i=1}^n(\check X_{ij}-\E[\check X_{ij}]) > s\Big) &\lesssim p\sup_{1 \leq j \leq p}\E\Big[ \Big(n^{-1/2}\sum_{i=1}^n(\check X_{ij}-\E[\check X_{ij}])\Big)^{J/2}\Big]/s^{J/2}\\
         &\lesssim p\sup_{1 \leq j \leq p}\sup_{1 \leq i \leq n}\E\Big[|\check X_{ij}-\E[\check X_{ij}]|^{J/2+\delta/2}\Big]^{J/(J+\delta)}/s^{J/2}\\
         &\lesssim p\sup_{1 \leq j \leq p}\sup_{1 \leq i \leq n}\left(\norm{\check X_{ij}}_{J/2+\delta/2}^{J/2}\right)/s^{J/2}\\
         &\lesssim n^{-J/2}s^{-J/2}
    \end{align}
    where the first inequality follows by the union bound and the Markov inequality, the second by equation 4.1 and the arguments in the proof of Theorem 3 from \cite{Yoshi78} and the third  due to Hölder's inequality. The forth inequality follows by a combination of Markov's and Hölder's inequalities.\\
    Next we  combine \eqref{p20}, \eqref{p21} and Lemma \ref{pl1} to obtain that
    \begin{align}   
        &\Big|\p\Big( \sup_{1 \leq j \leq p}n^{-1/2}\sum_{i=1}^n(\tilde X_{ij}-E[\tilde X_{ij})] \leq t-\sup_{1 \leq j \leq p} n^{-1/2}\sum_{i=1}^n(\check X_{ij}-\E[\check X_{ij}]) \Big)\\
         & \quad \quad -\p\Big( \sup_{1 \leq j \leq p}n^{-1/2}\sum_{i=1}^n(\tilde X_{ij}-E[\tilde X_{ij})] \leq t\Big)\Big| \\
        &        \lesssim \Big(\frac{(np)^{2/J}\log(pn)^{10+2/J}}{n}\Big)^{1/6}+s\sqrt{\log(pn)}+n^{-J/2}s^{-J/2}
    \end{align}
    so that we can use $\eqref{p20}$ together with the choice $s=n^{-1/2}$ to obtain
    \begin{align}
         \p\Big( \sup_{1 \leq j \leq p} n^{-1/2}\sum_{i=1}^nX_{ij} \leq t\Big) \geq \p\Big( \sup_{1 \leq j \leq p}n^{-1/2}\sum_{i=1}^n(\tilde X_{ij}-E[\tilde X_{ij})] \leq t\Big)+ o(1)
    \end{align}
    A similar arguments yields the reverse inequality and we therefore obtain
    \begin{align}
    \label{p30}
         \sup_{t\geq x_0}\Big|\p\Big( \sup_{1 \leq j \leq p} n^{-1/2}\sum_{i=1}^nX_{ij} \leq t\Big) - \p\Big( \sup_{1 \leq j \leq p}n^{-1/2}\sum_{i=1}^n(\tilde X_{ij}-E[\tilde X_{ij})] \leq t\Big)\Big| = o(1)
    \end{align}

    \textbf{Step 3: Applying Theorem \ref{pt2} to the truncated data}\\     

    We now apply Theorem \ref{pt2} to $\tilde X_i-\E[\tilde X_i]$ so that we may assume $D_n=(np)^{1/J}\log(n)$, using (D3) to verify 1) in Theorem \ref{pt2} yields
    \begin{align}
    \label{p31}
        &\sup_{t \geq x_0}\left|\p\Big( \sup_{1 \leq j \leq p}n^{-1/2}\sum_{i=1}^n(\tilde X_{ij}-E[\tilde X_{ij})] \leq t\Big)- \p^*\left( \sup_{1 \leq j \leq p}\frac{1}{\sqrt{nl}}\sum_{i=1}^{n-l+1}\nu_l\sum_{k=0}^{l-1}(\tilde X_{(i+k)j}-\E[\tilde X_{(i+k)j}])\leq t \right)\right|\\
        &=o(1)
    \end{align}

    Combining the three steps finishes the proof.
  
\end{proof}

\bibliographystyle{apalike}
\setlength{\bibsep}{2pt}
\begin{small}
\bibliography{main}
\end{small}

\end{document}